\documentclass[reqno,12pt]{amsart}
\usepackage{amsmath, amssymb, amsthm, amsfonts} 
\usepackage[english]{babel}
\usepackage{bbm}
\usepackage{graphicx}
\usepackage{url}
\usepackage{soul}
\usepackage{epstopdf}
\usepackage[ruled,vlined]{algorithm2e}
\usepackage[a4paper,bindingoffset=0.5cm,left=2cm,right=2cm,top=2.5cm,bottom=2cm,footskip=.8cm]{geometry}
\usepackage{rotating}
\usepackage{amsbsy,enumerate}
\usepackage{comment}
\usepackage{mathrsfs} 

\newcommand{\bs}{\boldsymbol}

\newcommand{\vb}{\vspace{3.2mm}}
\renewcommand{\hat}{\widehat}

\DeclareMathOperator*{\argmin}{arg\,min}

\newcommand{\vertiii}[1]{{\left\vert\kern-0.25ex\left\vert\kern-0.25ex\left\vert #1 \right\vert\kern-0.25ex\right\vert\kern-0.25ex\right\vert}}
\allowdisplaybreaks
\usepackage{hyperref}

\allowdisplaybreaks

\newtheorem{lemma}{Lemma}

\newtheorem{theorem}{Theorem}
\newtheorem{remark}{Remark}

\newtheorem{proposition}{Proposition}

\usepackage[utf8]{inputenc}
\usepackage{type1cm}         
\usepackage{multicol}        
\usepackage[bottom]{footmisc}

\usepackage{newtxtext}       %
\usepackage{newtxmath}       

\usepackage{tikz}
\usetikzlibrary{plotmarks}
\usetikzlibrary{decorations.pathreplacing}
\usepackage{pgfplots}
\pgfplotsset{width=10cm,compat=1.9}
\usepackage{subcaption}
\usepackage{standalone}
\usetikzlibrary{automata,positioning}
\usetikzlibrary{decorations.pathreplacing,decorations.markings,shapes.geometric}
\usetikzlibrary{shapes,arrows}
\usetikzlibrary{backgrounds,calc,positioning}

\usepackage{pgfplots}
\pgfplotsset{compat=1.9}

\usetikzlibrary{chains,shapes.multipart}
\usetikzlibrary{shapes,calc}
\usetikzlibrary{automata,positioning}

\usepackage{tikz}
\usetikzlibrary{automata,positioning}
\usetikzlibrary{decorations.pathreplacing,decorations.markings,shapes.geometric}
\usetikzlibrary{shapes,arrows}
\usetikzlibrary{backgrounds,calc,positioning}

\begin{document}

	\title[Estimating Graph Dynamics from Population Observations]{Estimating Graph Dynamics from Population Observations}
\author[P. Braunsteins, M. Mandjes,  {\tiny and} F. Montalescot]{Peter Braunsteins, Michel Mandjes,  {\tiny and} Florian Montalescot}
	
	\begin{abstract}
		In this paper we consider a population process evolving on a dynamic random graph. The dynamic random graph is an Erd\H{o}s--R\'enyi graph that is resampled every time unit, independently of the previous ones, with `edge existence probability' $p$. 
       The population process consists of $M$ individuals which reside at the vertices of the dynamic graph.
        At each point in time any of the $M$ individuals, supposing it resides at a vertex with $k$ neighbors, jumps to an adjacent vertex with probability $k/(k+1)$ (where this adjacent vertex is picked uniformly at random), and with probability $1/(k+1)$ it stays where it is. 
  We suppose we observe the numbers of individuals at each of the {vertices}, but not the evolving random graph itself. We propose two estimators for $p$, and establish their consistency and asymptotic normality.

\vb


\vb

\noindent
{\sc Affiliations.} 
PB is with 
School of Mathematics and Statistics, Anita B. Lawrence Centre, The University of New South Wales, Sydney NSW 2052, Australia. 

\vb

\noindent 
MM is with the Mathematical Institute, Leiden University, P.O. Box 9512,
2300 RA Leiden,
The Netherlands. He is also affiliated with Korteweg-de Vries Institute for Mathematics, University of Amsterdam, Amsterdam, The Netherlands; E{\sc urandom}, Eindhoven University of Technology, Eindhoven, The Netherlands; Amsterdam Business School, Faculty of Economics and Business, University of Amsterdam, Amsterdam, The Netherlands.

\vb

\noindent 
FM is with École Polytechnique, Avenue René Descartes, Palaiseau 91120, France. 

\vb

\noindent
{\sc Acknowledgments.} 
Date: {\it \today}.

\vb

\noindent
{\sc Email.} {\scriptsize \url{p.braunsteins@unsw.edu.au}, \url{m.r.h.mandjes@math.leidenuniv.nl}, \url{florian.montalescot@polytechnique.edu}}

	\end{abstract}

	\maketitle
 
\section{Introduction}

The bulk of the random graph literature is focused on \emph{static} models. That is, the emphasis is on analyzing the structural properties of a random graph at \emph{a fixed point in time}.  A central example is the Erd\H{o}s--R\'enyi random graph \cite{ER}, in which each potential edge between a pair of nodes is included independently with probability $p\in(0,1)$. This model serves as a foundational framework for understanding phase transitions and connectivity phenomena in random graphs. Typical questions include the presence and size of a giant component, the distribution of degrees, or the number of connected components. We refer the reader to \cite{RvdH1, RvdH2} for a modern account. 


There are strong practical motivations to consider random graphs that are {\it stochastically evolving} in order to capture the shifting structure and the dynamic nature of real-world networks. These models are known as {\it temporal random graphs} \cite{HolmeSaramaki2012}, and lead to novel mathematical challenges and insights; see e.g.\  \cite{ATH,CRA}. Indeed, many classical results in the static context can be reformulated in dynamic settings. For example, the sample-path large deviation principle from \cite{BRA} builds on and extends the foundational result for static Erd\H{o}s--Rényi graphs given in \cite{ChatterjeeVaradhan2011}. Similarly, the work in \cite{hkm25a} on the evolution of the largest eigenvalue extends the findings in \cite{ErdosKnowlesYauYin2013} on static graphs.


The next natural extension is to define a stochastic process \textit{on} the dynamic random graph. Indeed, in virtually any application area, the assumption of a fixed underlying network is unrealistic. In epidemiology, for instance, contact networks evolve as individuals change their behavior in response to an outbreak. Furthermore, it has been shown that epidemic models that account for the dynamic nature of the network exhibit novel behavior \cite{BBHM26,BBLS19,BB22}. In social networks, links between users appear and disappear over time, influenced by external events or internal platform dynamics. Again, its been shown that voter models that account for the dynamic nature of social ties display unique properties \cite{BBHM24,BS17,D12}. Further examples include financial networks that exhibit time-varying interbank exposures and communication networks where the topology may change as devices connect or disconnect. These examples highlight the need for models that incorporate both network evolution and processes defined on the network, to more accurately assess system behavior.


To quantitatively assess a {\it doubly stochastic} network of the type described above, a central challenge is to estimate the model parameters. While there is a substantial literature on statistical estimation in static \cite{CD13} and dynamic \cite{KH14} networks, far less has been done for doubly stochastic networks. When {\em both} the network and the trajectories of individuals moving across the network are observed continuously, we can often apply standard statistical techniques such as likelihood estimation (see for instance \cite[Section 3.2]{BAXV22}). However, in practice we typically observe only the process evolving on the network, rather than the dynamic network itself. For instance, during an epidemic we may observe the time series of the total number of infections, but not the evolution of the underlying social contact network through which the disease spreads. Similarly, while individuals are often surveyed about their opinions on a given topic, the evolution of the social connections that shape those opinions is rarely observed.
To the best of the authors’ knowledge, the literature does not yet contain an attempt to estimate the parameters of a dynamic network solely from observations of a stochastic process evolving on that network.


In this paper, we introduce a simple doubly stochastic model, and focus on estimating the parameters of the evolving random graph solely from observations of the individuals traversing it, i.e., without having direct access to observations of the network itself.
We deliberately keep the model as simple as possible to highlight that even in very basic settings, a variety of subtle effects can emerge. Most notably, although the individuals do not interact directly, their movement over a shared, evolving network creates a nontrivial dependence structure among them. More specifically, the network is an Erd\H{o}s--R\'enyi random graph with edge probability $p$ which is resampled independently at each discrete time point $t=1,\dots,T$, and the process on the graph consists of random walkers who either to remain at their current node or to move to a neighboring node based on the current graph structure. We suppose we observe the total number of walkers at each vertex at every time point. Based on this information we construct two consistent and asymptotically normal estimators for $p$ as $T \to \infty$.

To place the contribution of this paper in a broader perspective, we relate our work to the literature on {\it inverse problems}, that is, problems in which the goal is to infer underlying model parameters from limited or indirect data. A classic example involves estimating the offspring distribution of a Galton--Watson branching process from total population size counts. Indeed, in this setting it has been shown that {\em only} the first two moments of the offspring distribution can be estimated consistently \cite{GUT,L82}, which highlights the mathematical subtleties which can occur. Related inverse problems also arise in finance \cite{ACL, DUF}, epidemics \cite{CAU}, and queueing theory \cite{HAN, RAV}. More closely related to this paper is \cite{BAXV22}, where partial observation of {\em both} the network and the process on the network is used to estimate the parameters of a doubly stochastic network, and \cite[Section 6.3]{B20} (and references therein), where the transmission tree in an epidemic model is used to estimate the parameters of an unobserved static network of social connections. These papers do not attempt to establish asymptotic properties of their estimators such as consistency. Consistent estimators are established in \cite{Man_Wan}, where the on and off rates of the edges in a dynamic Erd\H{o}s--R\'enyi graph are estimated using observed subgraph counts. 


This paper is organized as follows. Section \ref{S2} introduces the model and states the main objective of the paper. In Section \ref{MOM}, we employ the method of moments to construct an estimator $\hat p_T$, obtained by matching a temporal covariance associated with the individuals’ positions to its empirical counterpart. The asymptotic normality of this estimator is established in Section \ref{S4}. In Section \ref{S5}, we propose a least-squares estimator, denoted by $\bar p_T$, and prove that it is asymptotically normal as well. Finally, Section \ref{S6} presents numerical experiments, including a comparison of the performance of the two estimators.

\section{Model and objective}\label{S2}

Our model consists of a dynamic random graph and random walkers who move through the dynamic graph. It is characterized by the following discrete-time dynamics:
\begin{itemize}
\item[$\circ$] {\em Graph dynamics:} Every at each time unit $t=1, \dots, n$ an Erd\H{o}s--R\'enyi random graph with edge probability $p$ is sampled independently of everything else (including the previous random graph). This means that at every point in time an edge is present between a vertex-pair $(i,j)$ with probability $p$ independently of all other edges; here $i,j\in\{1,\ldots,n\}$ with $i<j$ and $n\in\{2,3,\ldots\}$.
\item[$\circ$] {\em Walker dynamics:} There are a total of $M\in{\mathbb N}$ walkers moving through the graph. Suppose that at time $t$, an individual is located at vertex $i\in\{1,\ldots,n\}$, and that this vertex has $k\in\{0,\ldots,n-1\}$ neighbors. Then, with probability $k/(k+1)$, the individual moves to one of these neighbors (chosen uniformly at random), and with probability $1/(k+1)$, it remains at the current vertex. Each individual at vertex $i$ acts independently according to this rule. In what follows, the random variable $M_{i,t}$ denotes the number of individuals at vertex $i\in{1,\ldots,n}$ at time $t\in{\mathbb N}$. We use the compact notation ${\bs M}_t \equiv (M_{1,t}, \ldots, M_{n,t})$.
\end{itemize}
A visual representation of the process at three consecutive time units is given in Figure\ \ref{F1}.



\begin{figure}
\begin{center}
\begin{tikzpicture}[scale=1.05]

\def\nodes{
  \coordinate (A) at (0,0);
  \coordinate (B) at (1,1.5);
  \coordinate (C) at (2.2,1.2);
  \coordinate (D) at (1.2,-0.6);
  \coordinate (E) at (2.4,-0.4);
  \coordinate (F) at (0,1.2);
}

\newcommand{\nodewithdots}[2]{%
  \filldraw[white, draw=black, thick] (#1) circle (7.5pt);
  \pgfmathtruncatemacro{\ndots}{#2}
  \ifnum\ndots>0
    \def\r{2pt}
    \foreach \i in {1,...,\ndots}{
      \pgfmathsetmacro{\angle}{360/\ndots*(\i-1)}
      \path (#1) ++(\angle:\r) node[fill=black, circle, inner sep=0.9pt]{};
    }
  \fi
}

\begin{scope}[xshift=0cm]
  \nodes
  \draw (A) -- (B);
  \draw (B) -- (C);
  \draw (A) -- (D);
  \draw (D) -- (B);
  \draw (C) -- (E);
  \draw (A) -- (F);
  \nodewithdots{A}{3}
  \nodewithdots{B}{3}
  \nodewithdots{C}{1}
  \nodewithdots{D}{0}
  \nodewithdots{E}{2}
  \nodewithdots{F}{1}
\end{scope}

\begin{scope}[xshift=4cm]
  \nodes
  \draw (A) -- (B);
  \draw (B) -- (C);
  \draw (C) -- (E);
  \draw (D) -- (E);
  \nodewithdots{A}{2}
  \nodewithdots{B}{2}
  \nodewithdots{C}{2}
  \nodewithdots{D}{2}
  \nodewithdots{E}{1}
  \nodewithdots{F}{1}
\end{scope}

\begin{scope}[xshift=8cm]
  \nodes
  \draw (A) -- (D);
  \draw (D) -- (B);
  \draw (C) -- (E);
  \draw (D) -- (E);
  \draw (A) -- (F);
  \nodewithdots{A}{2}
  \nodewithdots{B}{3}
  \nodewithdots{C}{3}
  \nodewithdots{D}{1}
  \nodewithdots{E}{0}
  \nodewithdots{F}{1}
\end{scope}

\end{tikzpicture}
\end{center}
\caption{\label{F1}Our dynamic network, together with the individuals moving on it, shown at three successive time points ($n=6$, $M=10$).}
\end{figure}
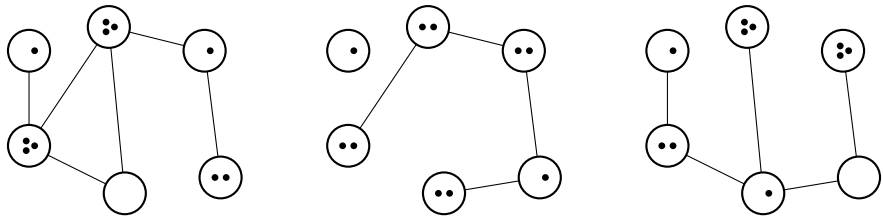

Our goal is to estimate the `edge existence probability' $p$ using only ${\bs M}_t$ for $t=1,\ldots,T$, i.e., the number of walkers at each vertex at each time.  
Importantly, we work with a very limited amount of information: neither the evolving random graph itself is not observed, nor the movements of the individuals. Throughout we assume that we observe the system {\it in stationarity}.

\begin{remark}
  {\em In the random graph literature, it is common to study the asymptotic regime where the number of vertices $n$ tends to infinity, often with $p \equiv p_n = \alpha/n$ for some $\alpha > 0$. In contrast, our analysis fixes $n$ and considers the asymptotic performance of the estimator as the number of observations $T$ grows large.
 However, in Remark \ref{R3} we briefly comment on the (classical) setting where $p \equiv p_n = \alpha/n$ for some $\alpha > 0$ and $n \to \infty$. 
  }\hfill $\Diamond$
\end{remark}

\section{Method-of-moments based estimator}\label{MOM}
In this section we define an estimator of $p$ based on the method of moments. In the first subsection we derive several auxiliary results, in the second subsection we evaluate second moments pertaining to the population vector ${\bs M}_t$, while in the third subsection we define our estimator. 

\subsection{Auxiliary results}
A first useful step concerns 
conditioning on the individual counts ${\bs M}_t$, and then computing the conditional mean ${\mathbb E}\big[M_{i,t+1}\,|\, {\bs M}_t\big].$
Let $Y_{ij}$ be the number of individuals hopping from vertex $i$ to vertex $j$ in time slot $t+1$, so that
\[M_{i,t+1} = M_{i,t} - \sum_{j\not=i} Y_{ij} + \sum_{j\not=i} Y_{ji} .\]
Define
\begin{align}
    F_n(p)&:=\frac{1-(1-p)^{n}}{np},\label{Fn}\quad
    G_n(p):=\frac{1-F_n(p)}{n-1}.
\end{align}
Since the index \( n \) is kept fixed throughout, we often omit it from the notation for convenience and simply write \( F(p) \) and \( G(p) \) in what follows. Note that $F(p)$ can be interpreted as the probability that the individual stays at the same vertex, and $G(p)$ can be interpreted as the probability the individual moves to a specific vertex $j$ different from its current vertex $i$.
\begin{lemma} \label{L1} For $i\in\{1,\ldots,n\}$ and $t\in{\mathbb N}$,
   \[{\mathbb E}\big[M_{i,t+1}\,|\, {\bs M}_t\big] = (F(p)-G(p)) \,M_{i,t}+ G(p)\,M.\] 
\end{lemma}
{\it Proof.} Observe that, recalling that the number of neighbors of vertex $i$ has a binomial distribution with parameters $n-1$ and $p$, in self-evident notation (and with the two binomial random quantities in the middle expression being independent),
\[{\mathbb E}\left[M_{i,t} - \sum_{j\not=i} Y_{ij}\,\Big\vert\,{\bs M}_t\right]= {\mathbb E}\left[ {\rm Bin}\left(M_{i,t},\frac{1}{{\rm Bin}(n-1,p)+1}\right)\right]=M_{i,t}\,{\mathbb E}\left[\frac{1}{{\rm Bin}(n-1,p)+1}\right]. \]
Then, relying on standard calculations involving Newton's binomium, 
\begin{align*}
  {\mathbb E}\left[\frac{1}{{\rm Bin}(n-1,p)+1}\right] &=
  \sum_{k=0}^{n-1} \frac{1}{k+1}\binom{n-1}{k} p^k (1-p)^{n-1-k}\\
  &=\frac{1}{n} \sum_{k=0}^{n-1}  \binom{n}{k+1}p^k(1-p)^{n-1-k}=\frac{1}{np}\sum_{\ell=1}^{n}  \binom{n}{\ell}p^\ell(1-p)^{n-\ell}\\
  &= \frac{1-(1-p)^{n}}{np}=F(p),
\end{align*}
where $F(p)$ was defined in \eqref{Fn}.
Also, for $i\not=j$,
\begin{align*}
{\mathbb E}\big[  Y_{ji}\,\big\vert\,{\bs M}_t\big] &= M_{j,t}\,\frac{1-F(p)}{n-1};
\end{align*}
indeed, $1-F(p)$ is the probability of the individual jumping from vertex $j$ to another vertex, and then with probability $1/(n-1)$ it jumps to vertex $i$.

Combining the above calculations, we thus conclude that
\begin{align*}
  {\mathbb E}\big[M_{i,t+1}\,|\, {\bs M}_t\big]&=  F(p)\,M_{i,t} +\frac{1-F_n(p)}{n-1}\sum_{j\not=i} M_{j,t}.
\end{align*}
Now use that $\sum_{j\not=i} M_{j,t}=M-M_{i,t}$, and recall the definition of $G(p)$ given in \eqref{Fn}, so as to obtain the stated claim. \hfill $\Box$

\begin{remark}\label{R2}
    {\em Observe that
     \begin{align*}\sum_{i=1}^n{\mathbb E}\big[M_{i,t+1}\,|\, {\bs M}_t\big] &= (F(p)-G(p)) \sum_{i=1}^n\ M_{i,t}+ n \,G(p)\,M\\
     &= M\big(F(p) +(n-1)G(p)\big) = M, 
     \end{align*}
     as desired. } \hfill$\Diamond$
\end{remark}

\begin{remark}\label{R3}{\em 
As mentioned, $F(p)$ is the probability that, during one timestep, an individual stays at the same vertex. It is seen that $F(\cdot)$ is decreasing, with $F(0)=1$ (the individuals with overwhelming probability do not jump) and $F(1)= 1/n$ (the individuals 
jump to a vertex that is chosen uniformly at random).

We now consider the impact of the number of vertices $n$. 
Observe that, as $n\to\infty$, \[nF_n(p)={\mathbb E}\left[\frac{n}{{\rm Bin}(n-1,p)+1}\right] \to \frac{1}{p},\] as anticipated.  Interestingly, when considering the parametrization $p=\alpha/n$ and taking $n\to\infty$, we have 
\[\lim_{n\to\infty}F_n(\alpha/n) = \frac{1-e^{-\alpha}}{\alpha}.\]
This is a `{slow-moving model}' in which individuals have an asymptotically positive probability of remaining at their current node as $n \to \infty$.} \hfill$\Diamond$
\end{remark}

The main idea behind our approach is that we can learn the value of $p$ by exploiting the temporal correlation of the ${\bs M}_t$ vectors.
Therefore, the next step towards our method-of-moments estimator is to evaluate the covariance between $M_{i,t}$ and $M_{i,t+1}$ in terms of the unknown parameter $p$. We recall that we assume that the system is in stationarity.
Applying the tower property in combination with Lemma~\ref{L1}, and with $F(p)$ and $G(p)$ as given in \eqref{Fn}, 
\begin{align}c(p)&:={\mathbb C}{\rm ov}(M_{i,t},M_{i,t+1}) = 
{\mathbb E}\Big[{\mathbb E}\big[M_{i,t+1}\,|\,M_{i,t}\big] \,M_{i,t}\Big] - {\mathbb E}\big[M_{i,t}\big] {\mathbb E}\big[M_{i,t+1}\big]\notag \\
&=
(F(p)-G(p))\,{\mathbb E}\big[M_{i,t}^2\big] + G(p) \,M\,{\mathbb E}\big[M_{i,t}\big] -  {\mathbb E}\big[M_{i,t}\big] {\mathbb E}\big[M_{i,t+1}\big]\notag\\
&=
(F(p)-G(p))\,{\mathbb E}\big[M_{i,t}^2\big] + G(p) \,\frac {M^2}{n} - \frac{M^2}{n^2}  \label{cp}
; \end{align}
in the last equality it has been used that ${\mathbb E}\big[M_{i,t}\big]=M/n$ for any $t$.
\begin{remark}\label{R4}{\em 
It is clear that $M_{i,t}$ can be expressed as the sum of $M$ Bernoulli random variables, each with success probability $1/n$  (recalling that we observe the system in stationarity). Importantly, however, $M_{i,t}$ does {\it not} have a binomial distribution, as the $M$ individual trials are not independent. If at some point in time vertex $i$ is an `attractive' (i.e., relatively likely) location for one of the individuals, it is also attractive for the other $M-1$ individuals, simply because all individuals are confronted with the same evolution of the dynamic random graph. 
As a consequence, we indeed have that 
${\mathbb E}\big[M_{i,t}\big]=M/n$ for any $t$, but it does {\it not} hold that  ${\mathbb E}\big[M^2_{i,t}\big]=(M/n)(1-1/n).$ }
\hfill$\Diamond$
\end{remark}

\subsection{Second moment}
Our next goal, in order to derive an expression for $c(p)$, is to express the second moment ${\mathbb E}\big[M^2_{i,t}\big]$ in terms of $p$; see Eqn.\ \eqref{cp}.
This formula is stated at the end of this subsection in Proposition~\ref{P1} and relies on the notation introduced in the display equations below.

Let $X_{m,t}\in\{1,\ldots,n\}$ denote, for $t\in {\mathbb N}$ and $m\in\{1,\ldots,M\}$, the location of the $m$-th individual, so that
\begin{equation}\label{decompo}
    M_{i,t} = \sum_{m=1}^M {\bs 1}\{X_{m,t}=i\},
\end{equation}
for $i\in\{1,\ldots,n\}.$
To this end, we first determine the probabilities, for given locations $i,j\in\{1,\ldots,n\}$ such that $i\neq j$,
\[ \Pi_=(p\,|\,i):={\mathbb P}(X_{1,t}=X_{2,t}=i),\quad \Pi_\neq(p\,|\,i,j):={\mathbb P}(X_{1,t}=i,X_{2,t}=j).\]
Observe that these probabilities do not depend on $t$ (due to our stationarity assumption). In addition, they do not depend on $i,j$ (due to the symmetries in the model), motivating why we below sometimes simply write $\Pi_=(p)$ and $\Pi_\neq(p)$ when the location is not relevant. Also, in the definitions of $\Pi_=(p)$ and $\Pi_\neq(p)$ we focused on the locations of individuals $1$ and $2$, but clearly these equal their counterparts for individuals $m$ and $m'$ for any $m,m'$ such that $m\neq m'$.

It is clear that 
\[1=\sum_{i=1}^n \sum_{j=1}^n \mathbb{P}(X_{1,t}=i, X_{2,t}=j)=n\,\Pi_=(p) + n(n-1)\,\Pi_\neq(p),
\]
which implies 
\begin{equation}\label{Eqn:Ms2}
\Pi_\neq(p)=\frac{1-n\Pi_=(p)}{n(n-1)}=\frac{1}{n(n-1)}-\frac{\Pi_=(p)}{n}.
\end{equation}
As a consequence, in order to identify $\Pi_=(p)$ and $\Pi_\neq(p)$ we need one additional relation between them.
We obtain this additional relation by performing a one-step argument. 
To this end we first compute four probabilities.  Define 
\[b_n(k,p):={\mathbb P}({\rm Bin}(n,p)=k).\]
\begin{itemize}
    \item[$\circ$]
{\it Scenario 1}. Define $\pi_1(p)$ as the probability that individuals $1$ and $2$ are at vertex $i$ at time $t+1$, given that they were already at vertex $i$ at time $t$. If $i$ has $k$ neighbors at time $t+1$, this probability is $(k+1)^{-2}$. We thus obtain that
\[\pi_1(p)= \sum_{k=0}^{n-1} b_{n-1}(k,p) \left(\frac 1 {k+1}\right)^2.\]
\item[$\circ$] 
{\it Scenario 2}. Define $\pi_2(p)$ as the probability that individuals $1$ and $2$ are at vertex $i$ at time $t+1$, given that individual $1$ was already at vertex $i$ at time $t$ but individual 2 was at vertex $j$. This requires the edge between $i$ and $j$ to exist at time $t+1$, which has probability $p$. If vertex $i$ has $k$ additional neighbors at time $t+1$, individual $1$ stays there with probability $(k+2)^{-1}$. If vertex $j$ has $k$ additional neighbors at time $t+1$, individual $2$ jumps to vertex $i$ with probability $(k+2)^{-1}$. We conclude that
\[\pi_2(p) = p\cdot \left(\sum_{k=0}^{n-2} b_{n-2}(k,p)\frac{1}{k+2}\right)^2.\]
\item[$\circ$] 
{\it Scenario 3}. Define $\pi_3(p)$ as the probability that individuals $1$ and $2$ are at vertex $i$ at time $t+1$, given that both individuals were at vertex $j$ at time $t$. This requires the edge between $i$ and $j$ to exist at time $t+1$, which has probability $p$. If vertex $j$ has $k$ additional neighbors at time $t+1$, both individuals jump to vertex $i$ with probability $(k+2)^{-2}$. Hence,
\[\pi_3(p) = p\cdot\sum_{k=0}^{n-2} b_{n-2}(k,p)\left(\frac{1}{k+2}\right)^2.\]
\item[$\circ$] {\it Scenario 4}.  Finally, define $\pi_4(p)$ as the probability that individuals $1$ and $2$ are at vertex $i$ at time $t+1$, given that  individual $1$ was at vertex $j$ and individual $2$ was at vertex $j'\not\in\{i,j\}$ at time $t$. This requires that both the edge between $i$ and $j$ and the edge between $i$ and $j'$ exist at time $t+1$, which occurs with probability $p^2$. If vertex $j$ has $k$ additional neighbors at time $t+1$, individual $1$ jumps to vertex $i$ with probability $(k+2)^{-1}$. If vertex $j'$ has $k$ additional neighbors at time $t+1$, individual $2$ jumps to vertex $i$ with probability $(k+2)^{-1}$. We find
\[\pi_4(p) = p^2\cdot \left(\sum_{k=0}^{n-2} b_{n-2}(k,p)\frac{1}{k+2}\right)^2.\]
\end{itemize}

The one-step argument then yields
\begin{align*}
{\mathbb P}(X_{1,t+1}=X_{2,t+1}=i)=&\:{\mathbb P}(X_{1,t}=X_{2,t}=i) \,\pi_1(p)\:+\\
&\: 2(n-1)\, {\mathbb P}(X_{1,t}=i,X_{2,t}=j) \,\pi_2(p)\:+\\
&\: (n-1) \, {\mathbb P}(X_{1,t}=j,X_{2,t}=j) \,\pi_3(p)\:+\\
&\: (n-1)(n-2)\, {\mathbb P}(X_{1,t}=j,X_{2,t}=k) \,\pi_4(p).
\end{align*}
The factors $2(n-1)$, $n-1$, and $(n-1)(n-2)$, corresponding to Scenarios 2, 3, and 4, respectively, can be explained as follows. 
\begin{itemize}
\item[$\circ$] We defined $\pi_2(p)$ in terms of the event that individual $2$ jumps from some vertex $j\neq i$ to vertex $i$, which yields the factor $n-1$. Realize that there is the counterpart of this event in which individual $1$ has jumped from some vertex $j\neq i$ to vertex $i$. We thus find $2(n-1).$
\item[$\circ$] We defined $\pi_3(p)$ in terms of the event that both individual $1$ and individual $2$ jump from some vertex $j\neq i$ to vertex $i$, yielding a factor $n-1$.
\item[$\circ$] We defined $\pi_3(p)$ in terms of the event that  individual $1$ and individual $2$ jump from distinct vertices $j,j'\neq i$ to vertex $i$. This leads to a factor $(n-1)(n-2)$.
\end{itemize}

Using once more that the system is in stationarity, we arrive at the following identity: for ease abbreviating $\Pi_=\equiv \Pi_=(p)$ and $\Pi_\neq\equiv \Pi_\neq(p)$,
\[\Pi_= = \Pi_= \,\pi_1(p) + 2(n-1)\,\Pi_\neq\,\pi_2(p) +(n-1)\,\Pi_=\,\pi_3(p)+(n-1)(n-2)\Pi_\neq \,\pi_4(p).\]
This implies that $\Pi_=(p) = \kappa(p)\, \Pi_\neq(p)$, with
\begin{equation}\label{kappa}\kappa(p):=\frac{2(n-1)\,\pi_2(p)+(n-1)(n-2)\,\pi_4(p)}{1-\pi_1(p)-(n-1)\,\pi_3(p)}.\end{equation}
Combining this with \eqref{Eqn:Ms2}, we obtain
\[\Pi_\neq(p) = \frac{1-n\,\kappa(p)\,\Pi_\neq(p)}{n(n-1)}.\]
This, after elementary computations yields the stationary probability that two  individuals are at the same specified vertex, as well as the stationary probability that two  individuals are at two different specified vertices:
\[ \Pi_=(p) =\frac{\kappa(p)}{n(n-1+\kappa(p))}, \quad \Pi_\neq(p) = \frac{1}{n(n-1+\kappa(p))}.\]
The final step is to use \eqref{decompo} to evaluate ${\mathbb E}\big[M^2_{i,t}\big]$. Clearly, once more appealing to the model's intrinsic symmetries, we obtain
\begin{align*}
    {\mathbb E}\big[M^2_{i,t}\big]&=M \,{\mathbb E}\big[{\bs 1}\{X_{1,t}=i\}\big] + M(M-1) \,{\mathbb E}\big[{\bs 1}\{X_{1,t}=X_{2,t}=i\}\big]= \frac{M}{n}+M(M-1)\,\Pi_=(p).
\end{align*}

We have proven the following result.
\begin{proposition} \label{P1} For $i\in\{1,\ldots,n\}$ and $t\in{\mathbb N}$,
\[{\mathbb E}\big[M^2_{i,t}\big] = \frac{M}{n}+\frac{M(M-1)\,\kappa(p)}{n(n-1+\kappa(p))}.\]
\end{proposition}


\subsection{Estimator}
Upon combining Proposition \ref{P1} with Eqn.\ \eqref{cp}, we have found a closed-form expression for $c(p)={\mathbb C}{\rm ov}(M_{i,t},M_{i,t+1})$ in terms of the unknown `edge existence probability' $p$. 

\begin{theorem}\label{T1}
    For $n\in\{2,3,\ldots\}$,
    \[c(p) = \left(F(p)-G(p)\right)\left(\frac{M}{n}+\frac{M(M-1)\,\kappa(p)}{n(n-1+\kappa(p))}\right)+G(p)\frac{M^2}{n}-\frac{M^2}{n^2}.\]
\end{theorem}

Now that we have expressed $c(p)$ in terms of $p$ we can define an estimator that exploits the temporal correlation structure present in the process $({\bs M}_t)_t$. 
Having access to the $T$ observations ${\bs M}_1,\ldots, {\bs M}_{T}$, we may estimate $c_n(p)$ by the empirical covariance
\[\hat c_{T}:=\frac{1}{n}\sum_{i=1}^n\left( \frac{1}{T-1}\sum_{t=1}^{T-1}  N_{i,t}-\left(\frac{1}{T}\sum_{t=1}^TM_{i,t}\right)^2\right),\]
where $N_{i,t}:=M_{i,t}\,M_{i,t+1}.$
This suggests that we can estimate $p$ by $\hat p_T$, defined as the solution to the equation $\hat c_{T}=c(p)$, i.e., \[\hat p_T:=c^{-1}(\hat c_{T}).\]
The inverse $c^{-1}(\cdot)$ is well-defined, since $c(p)$ is a decreasing function of $p$. Informally, the higher $p$, the larger the number of neighbors, the higher the individuals' mobility, the lower the correlation between $M_{i,t}$ and $M_{i,t+1}$. The fact that $c(p)$ is monotone facilitates the numerical computation of $\hat p_T$ via an efficient bisection search. 

\section{Asymptotic normality} 
\label{S4}
The main objective of this section is to argue that the estimator $\hat p_T$ is asymptotically normal. To this end, we first establish that the $2n$-dimensional `sample-mean vector'
\[(\hat{\bs M}_T,\hat{\bs N}_{T-1})\equiv \left(\frac{1}{T}\sum_{t=1}^TM_{1,t},\ldots,\frac{1}{T}\sum_{t=1}^TM_{n,t},\frac{1}{T-1}\sum_{t=1}^TN_{1,t},\ldots,\frac{1}{T-1}\sum_{t=1}^TN_{n,t}\right)\]
is asymptotically normal, by which we mean that 
\begin{equation}
    \label{clt}
\sqrt{T}\left((\hat{\bs M}_T,\hat{\bs N}_{T-1})-{\mathbb E}\big[({\bs M}_0,{\bs N}_{0})\big]\right)\stackrel{\rm d}{\to} {\bs G}\end{equation}
where ${\bs G}$ is a $2n$-dimensional zero-mean normally distributed random vector. Then we argue that the covariance estimator $\hat c_T$ is asymptotically optimal. We complete our argument by showing that this also implies the asymptotic normality of $\hat p_T.$

\subsection{Asymptotic normality of the sample-mean vector}
By the Cram\'er--Wold device, the claim \eqref{clt} follows once we establish that, for arbitrary vectors ${\bs \xi}, {\bs \zeta} \in {\mathbb R}^n$,
\begin{equation}\label{eq:clt_CW}
\sqrt{T} \left({\bs \xi}^\top\hat{\bs M}_T + {\bs \zeta}^\top \hat{\bs N}_T 
- {\bs \xi}^\top {\mathbb E}\big[{\bs M}_0\big] 
- {\bs \zeta}^\top {\mathbb E}\big[{\bs N}_0\big]\right)
\end{equation}
converges to a (univariate) zero-mean normally distributed random variable. 
Noting that the $2n$-dimensional process \[({\bs V}_t)_{t\in{\mathbb Z}}=\left(\begin{array}{c}{\bs M}_t\\ {\bs N}_t\end{array}\right)_{t\in{\mathbb Z}}\] is a finite-state, irreducible, stationary Markov chain, we can apply the central limit theorem of \cite[Theorem 17.0.1]{MT12} to obtain asymptotic normality of the expression in \eqref{eq:clt_CW} for every ${\bs \xi}, {\bs \zeta} \in {\mathbb R}^n$.

We have thus established \eqref{clt}.
In the sequel we let $S$ denote the $(2n\times 2n)$-dimensional covariance matrix pertaining to the $2n$-dimensional zero-mean normally distributed  random vector ${\bs G}$, i.e.,
\begin{align*}
S
&:=
\sum_{t=-\infty}^{\infty}
{\mathbb C}{\rm ov}({\bs V_0,{\bs V}_t}) =
{\mathbb V}{\rm ar}({\bs V_0})
+
\sum_{t=1}^{\infty}
\Big[
{\mathbb C}{\rm ov}({\bs V}_0,{\bs V}_t)
+
{\mathbb C}{\rm ov}({\bs V}_t,{\bs V}_0)
\Big],
\end{align*}
where we have  used the standard notation ${\mathbb V}{\rm ar}({\bs X}):={\mathbb E}[{\bs X}{\bs X}^\top]- {\mathbb E}[{\bs X}]{\mathbb E}[{\bs X}]^\top$ and ${\mathbb C}{\rm ov}({\bs X},{\bs Y}):={\mathbb E}[{\bs X}{\bs Y}^\top]- {\mathbb E}[{\bs X}]{\mathbb E}[{\bs Y}]^\top$, for ${\bs X},{\bs Y}\in {\mathbb R}^{2n}$.

\subsection{Asymptotic normality of the covariance estimator} \label{asnp} We define the function $\psi:{\mathbb R}^{2n}\to{\mathbb R}^n$ by, for ${\bs u},{\bs v}\in{\mathbb R}^n$,
\[\psi({\bs u},{\bs v}):= \frac{1}{n}\sum_{i=1}^n\left(v_i - u_i^2\right).\]
Note that 
\[\hat c_T = \psi\big(\hat{\bs M}_T,\hat{\bs N}_{T-1}\big).\]
Informally,  the delta method provides asymptotic normality of a smooth function of (not necessarily independent) asymptotically normal random variables. Indeed, in the regime that $T$ grows large, ignoring contributions that are $o(T^{-1/2})$, 
\[\hat c_T \stackrel{\rm d}{=}c(p) +\frac{1}{\sqrt{T}}\sum_{i=1}^nG_i\left.\frac{\partial \psi}{\partial u_i}\right|_{({\bs u},{\bs v})={\bs e}}+\frac{1}{\sqrt{T}}\sum_{i=1}^nG_{n+i}\left.\frac{\partial \psi}{\partial v_i}\right|_{({\bs u},{\bs v})={\bs e}},\]
writing for conciseness ${\bs e}:=({\mathbb E}[{\bs M}_0],{\mathbb E}[{\bs N}_0])$. Formally, we conclude that
\[\sqrt{T}\big(\hat c_T-c(p)\big)\stackrel{\rm d}{\to} G_c,\]
where $G_c$ is a zero-mean Gaussian random variable with variance
\[\sigma^2_c:={\mathbb V}{\rm ar}\left(\sum_{i=1}^nG_i\left.\frac{\partial \psi}{\partial u_i}\right|_{({\bs u},{\bs v})={\bs e}}+\sum_{i=1}^nG_{n+i}\left.\frac{\partial \psi}{\partial v_i}\right|_{({\bs u},{\bs v})={\bs e}}\right).\]
Observing that 
\[\frac{\partial \psi}{\partial u_i}=-\frac{2u_i}{n},\quad \frac{\partial \psi}{\partial v_i}=\frac{1}{n}, \]
we obtain that, where we follow the convention that empty sums are defined as zero, and recalling that $S$ is the covariance matrix corresponding to ${\bs G}$,
\begin{align*}
    \sigma^2_c=\:& \frac{4}{n^2}\sum_{i=1}^n S_{ii}\,\big({\mathbb E}[M_{0,i}]\big)^2+\frac{1}{n^2}\sum_{i=1}^n S_{n+i,n+i} + \frac{8}{n^2}\sum_{i=1}^n\sum_{j=1}^{i-1}S_{ij}\,{\mathbb E}[M_{0,i}]\,{\mathbb E}[M_{0,j}] \:+\\
    &\frac{2}{n^2}\sum_{i=1}^n\sum_{j=1}^{i-1}S_{i+n,j+n}-\frac{4}{n^2}\sum_{i=1}^n\sum_{j=1}^{i-1}S_{i,j+n}\,{\mathbb E}[M_{0,i}],
\end{align*}
where ${\mathbb E}[M_{0,i}]=M/n.$

\subsection{Asymptotic normality of the parameter estimator} We again apply the delta method to establish the asymptotic normality of our estimator $\hat p_T$. Informally, recalling that $\hat p_T=c^{-1}(\hat c_{T})$, we have, as $T\to\infty$ (again ignoring $o(T^{-1/2})$ terms),
\[\hat p_T \stackrel{\rm d}{=} c^{-1}\left(c(p) +\frac{1}{\sqrt T} G_c\right)\stackrel{\rm d}{=} p
+\frac{1}{c'(p)}\frac{1}{\sqrt T} G_c. \]
 In a formal form, we have thus established the following result.
\begin{theorem}\label{asnc}
As $T\to\infty$,  
\[\sqrt{T}\big(\hat p_T-p\big)\stackrel{\rm d}{\to} G_p,\]
where $G_p$ is a zero-mean random variable with variance $(c'(p))^{-2}\,\sigma^2_c$.
\end{theorem}

\begin{remark}
{\em One could improve the variance performance of the estimator slightly by using that one knows that ${\mathbb E}[M_{i,t}]=M/n$. Conretely, one could replace $\hat c_T$ by $\hat c_T'$, given by 
\[\hat c'_{T}:=\frac{1}{n}\sum_{i=1}^n\left( \frac{1}{T-1}\sum_{t=1}^{T-1}  N_{i,t}-\left(\frac{M}{n}\right)^2\right).\]
Using the same reasoning as above, one can argue that the resulting estimator for $p$ is asymptotically normal. $\hfill\Diamond$}
\end{remark}

\section{Least-squares based estimator}\label{S5}
In this section we consider an alternative to the method-of-moments based estimator in the previous sections. The main idea is to minimize the following sum of squares: 
\[
\bar p_T:=\argmin_{p \in [0,1]}\sum_{i=1}^n \sum_{t=1}^{T-1}  (M_{i,t+1}-{\mathbb E}\big[M_{i,t+1}\,|\, {\bs M}_t\big])^2.
\]
An attractive feature of this estimator is that it does not require any assumption about the process being in stationarity (unlike the method-of-moments based estimator $\hat p_T$). 
Write
\[H(p, M_{i,t}):={\mathbb E}\big[M_{i,t+1}\,|\, {\bs M}_t\big]=\left(\frac{nF(p)-1}{n-1}\right)M_{i,t}+\left(\frac{1-F(p)}{n-1}\right)M, \]
given by Lemma \ref{L1} (where we recall that ${\mathbb E}[M_{i,t+1}\,|\, {\bs M}_t]$ does not depend on $M_{j,t}$ for $j\not=i$). Considering the first order condition, the estimator $\bar p_T$ solves
\[\sum_{i=1}^n\sum_{t=1}^{T-1} (M_{i,t+1}-H(p, M_{i,t}))\,H'(p, M_{i,t})=0, \]
where $H'(p, M_{i,t})$ is the derivative of $H(p, M_{i,t})$ with respect to $p$. 
It turns out practical to rewrite this equation in a slightly different form, as follows:
\[H(p, M_{i,t}) = I(p)\,M_{i,t}+J(p)\,M,\quad I(p):=\frac{nF(p)-1}{n-1},\quad J(p):=\frac{1-I(p)}{n}.\]
Recall that we defined $N_{i,t}=M_{i,t}\,M_{i,t+1}$; in addition, let 
$N^\circ_{i,t}=M_{i,t}^2$. 
It takes some routine calculus to verify that the above equation can be rewritten as (repeatedly using that $\sum_{i=1}^n M_{i,t}=M$ for any $t$)
\[\frac{1}{T-1}\left(\alpha(p)\sum_{i=1}^n\sum_{t=1}^{T-1}  N_{i,t} + \beta(p)\sum_{i=1}^n\sum_{t=1}^{T-1} N_{i,t}^\circ \right) = \gamma(p),\]
where
\begin{align*}
    \alpha(p)&:= I'(p),\quad\quad  \beta(p):=  -I(p)\,I'(p),\\
    \gamma(p)&:= {M^2} \big(J'(p)-I(p)J'(p)-I'(p)J(p)-nJ(p)J'(p)\big)=\frac{M^2}{n}I'(p)(1-I(p)). 
\end{align*}
After dividing the full equation by $I'(p)$ and rearranging the terms, we find the equation
\[\left({\bs 1}^\top \hat {\bs N}_{T-1}-\frac{M^2}{n}\right) = I(p)\left({\bs 1}^\top \hat {\bs N}^\circ_{T-1}-\frac{M^2}{n}\right) ,\]
with the vectors $\hat{\bs N}_{T-1}$ and $\hat{\bs N}^\circ_{T-1}$ having the natural meaning.
Observe that this equation involves only the elementary quantity $I(p)$, and does not depend on the more involved object $\kappa(p)$  that was required to evaluate $\hat{p}_T$. It is further simplified to
\begin{equation}\label{Ip}I(p) = \frac{n \,{\bs 1}^\top \hat {\bs N}_{T-1}-{M^2}}{n \,{\bs 1}^\top \hat {\bs N}^\circ_{T-1}-{M^2}}.\end{equation}
Here it is observed that \eqref{Ip}
has a unique solution, because $I(p)$ inherits its monotonicity from $F(p)$; recall that from Remark \ref{R3} we know that $F(p)$ decreases in $p$. In addition, $I(0)=1$ and $I(1)=0$.
We can therefore write
\[\bar p_T= I^{-1}\left(\frac{n \,{\bs 1}^\top \hat {\bs N}_{T-1}-{M^2}}{n \,{\bs 1}^\top \hat {\bs N}^\circ_{T-1}-{M^2}}\right).\]

Analogous to our proof of the asymptotic normality of $\hat{p}_T$, we can now establish this property for $\bar{p}_T$ as well. As a first step we verify that, for any ${\bs \xi}$ and ${\bs\zeta}$,
\[\sqrt{T} \left({\bs \xi}^\top\hat{\bs N}_{T-1}+ {\bs\eta}^\top \hat{\bs N}^\circ_{T-1} -{\bs \xi}^\top{\mathbb E}\big[
{\bs N}_0\big]+ {\bs\eta}^\top {\mathbb E}\big[{\bs N}^\circ_0\big]\right)
\]
converges to a (univariate) normally distributed random variable; here ${\mathbb E}[
{\bs N}_0]$ and ${\mathbb E}[
{\bs N}^\circ_0]$ represent the steady-state expectations of ${\bs N}_0$ and ${\bs N}^\circ_0$, respectively. By the Cram\'er-Wold device this means that \begin{equation}
    \label{clt2}
\sqrt{T}\left((\hat{\bs N}_{T-1},\hat{\bs N}^\circ_{T-1})-{\mathbb E}\big[({\bs N}_0,{\bs N}^\circ_{0})\big]\right)\stackrel{\rm d}{\to} {\bs G}^\circ\end{equation}
where ${\bs G}^\circ$ is a $2n$-dimensional zero-mean normally distributed random vector, say with covariance matrix $S^\circ$. As before, by applying the delta method, this leads to  the asymptotic normality of the estimator $\bar p_T$. To this end, define
\[\psi^{\circ}({\bs u},{\bs v}) := \left(n\sum_{i=1}^nu_i-M^2\right)\left/\left(n\sum_{i=1}^nv_i-M^2\right)\right.\]
We thus obtain up to $o(1/\sqrt{T})$-terms, with ${\bs e}^\circ:= ({\mathbb E}[
{\bs N}_0],{\mathbb E}[
{\bs N}_0^\circ])$,
\[\hat I_T\stackrel{\rm d}{=}I(p) +\frac{1}{\sqrt{T}}\sum_{i=1}^nG^{\circ}_i\left.\frac{\partial \psi^{\circ}}{\partial u_i}\right|_{({\bs u},{\bs v})={\bs e}^{\circ}}+\frac{1}{\sqrt{T}}\sum_{i=1}^nG^{\circ}_{n+i}\left.\frac{\partial \psi^{\circ}}{\partial v_i}\right|_{({\bs u},{\bs v})={\bs e}^{\circ}}.\]
We thus have that
\[\sqrt{T}\big(\hat I_T-I(p)\big)\stackrel{\rm d}{\to} G_I,\]
where $G_I$ is a zero-mean Gaussian random variable with variance
\[\sigma^2_I:={\mathbb V}{\rm ar}\left(\sum_{i=1}^nG^{\circ}_i\left.\frac{\partial \psi^{\circ}}{\partial u_i}\right|_{({\bs u},{\bs v})={\bs e}^{\circ}}+\sum_{i=1}^nG^{\circ}_{n+i}\left.\frac{\partial \psi^{\circ}}{\partial v_i}\right|_{({\bs u},{\bs v})={\bs e}^{\circ}}\right);\]
this variance can be further evaluated in terms of the entries of $S^\circ$ in the same manner as we evaluated $\sigma^2_p$ in Subsection \ref{asnp}. 
Following the same argumentation as in the proof of Theorem \ref{asnc}, we arrive at the following result. 
\begin{theorem}\label{asnI}
As $T\to\infty$,  
\[\sqrt{T}\big(\bar p_T-p\big)\stackrel{\rm d}{\to} G_p^\circ,\]
where $G_p^\circ$ is a zero-mean random variable with variance $(I'(p))^{-2}\,\sigma^2_I$.
\end{theorem}

\begin{remark} {\em 
    In the limit as $p \downarrow 0$, all individuals tend to remain at their current vertex, so that $\hat {\bs N}_{T-1}$ and $\hat {\bs N}^\circ_{T-1}$ effectively coincide. This implies that $\bar p_T=I^{-1}(1)$, which equals $0$, as expected. Likewise, in the limit as $p\uparrow 1$, the random variables $M_{i,t}$ and $M_{i,t+1}$ become independent, each having expectation $M/n$ (cf.\ Remark \ref{R4}), so that by a standard version of the law of large numbers
    \[\frac{1}{T-1}\sum_{t=1}^{T-1}N_{i,t}  = \frac{1}{T-1}\sum_{t=1}^{T-1}M_{i,t}\,M_{i,t+1} \to \frac{M^2}{n^2}\]
    as $T\to\infty$,
    and hence $n \,{\bs 1}^\top \hat {\bs N}_{T-1}\to{M^2}$. We conclude that in this regime $\bar p_T=I^{-1}(0)$, which equals $1$, as it should.} \hfill$\Diamond$
\end{remark}

\begin{remark}
    {\em One may wonder about the relationship between the least-squares estimator $\bar p_T$, which solves \eqref{Ip}, and the method-of-moments procedure discussed in Section~\ref{MOM}.
As stated in Lemma \ref{L1}, we have for $i\in\{1,\ldots,n\}$ and $t\in{\mathbb N}$,
   \[{\mathbb E}\big[M_{i,t+1}\,|\, {\bs M}_t\big] = (F(p)-G(p)) \,M_{i,t}+ G(p)\,M,\]
   so that
   \begin{equation}\label{ID}{\mathbb E}\big[M_{i,t+1} M_{i,t}\big]= (F(p)-G(p))\,{\mathbb E}\big[M_{i,t}^2\big] +G(p)\,M\,{\mathbb E}\big[M_{i,t}\big].\end{equation}
   In Section \ref{MOM} we proceeded by finding an expression for ${\mathbb E}[M_{i,t}^2]$ in terms of $p$ (in Proposition \ref{P1}), so that $p$ could be estimated after having estimated the covariance between $M_{i,t}$ and $M_{i,t+1}$.

   An alternative approach, however, is to estimate $p$ directly from \eqref{ID}. Indeed, summing over $i\in\{1,\ldots,n\}$, $t\in\{1,\ldots,T-1\}$ and dividing by $T-1$ suggests the estimation equation
   \[{\bs 1}^\top\hat{\bs N}_{T-1} = (F(p)-G(p))\,{\bs 1}^\top \hat{\bs N}^\circ_{T-1}+G(p)\,M^2.\]
Using that $G(p)=(1-F(p))/(n-1)$, rearranging this equation directly yields \eqref{Ip}. }
\hfill$\Diamond$
\end{remark}

\section{Numerical experiments}\label{S6}
In the previous sections, we introduced the estimators $\hat p_T$ and $\bar p_T$. The primary objective of this section is to empirically evaluate their performance across a range of settings. Since it is not \emph{a priori} clear which of the two estimators will perform better, we systematically consider a sequence of scenarios. In each scenario, we select specific values for the parameters $p$, $n$, and $M$.

As established in Theorems \ref{asnc} and \ref{asnI}, the standard deviations of the estimators $\hat p_T$ and $\bar p_T$ are given by $(c'(p))^{-1} \sigma_c$ and $(I'(p))^{-1} \sigma_I$, respectively. These formulas highlight two key factors that determine the performance of an estimator. First, the steepness of the theoretical moment as a function of the parameter to be estimated ($p$, that is) plays a crucial role: the steeper the curve, the more sensitive the moment is to changes in the parameter, and the more precise the estimation. Second, the variability of the estimated quantity itself directly impacts performance: a smaller standard deviation leads to a more reliable estimator. Taken together, these two factors provide a clear framework for understanding and comparing the efficiency of $\hat p_T$ and $\bar p_T$.

We begin by presenting an experiment in which the parameters $M$ and $n$ are held fixed while $p$ is varied systematically. Recognizing that both ${\sigma_I}$ and ${\sigma_c}$ depend on $p$, we introduce the following quantities to facilitate comparison between the two estimators:
\[ \lambda(p):= \frac{c'(p)}{I'(p)}, \quad \mu(p):=\frac{\sigma_I(p)}{\sigma_c(p)}.\quad \nu(p):=\lambda(p)\,\mu(p).\]
Here, $\lambda(p)$ quantifies the relative sensitivity of the theoretical moments with respect to changes in $p$, while $\mu(p)$ captures the relative variability of the two estimated quantities. Together, they provide a clear framework for analyzing how estimator performance evolves as $p$ varies. Note that we have explicit expressions for $c'(p)$ and $I'(p)$, whereas ${\sigma_I}(p)$ and ${\sigma_c}(p)$ do not have closed-form expressions. Consequently, we estimate these latter quantities empirically by performing a large number of simulation experiments.

Before comparing the two estimators, we first present QQ-plots to empirically assess their asymptotic normality. Figures~\ref{F2}--\ref{F3} indicate that both $\hat p_T$ and $\bar p_T$ are well approximated by a normal distribution, consistent with our asymptotic results, as stated in Theorems \ref{asnc}--\ref{asnI}.

\begin{figure}
    \includegraphics[width=0.32\textwidth]{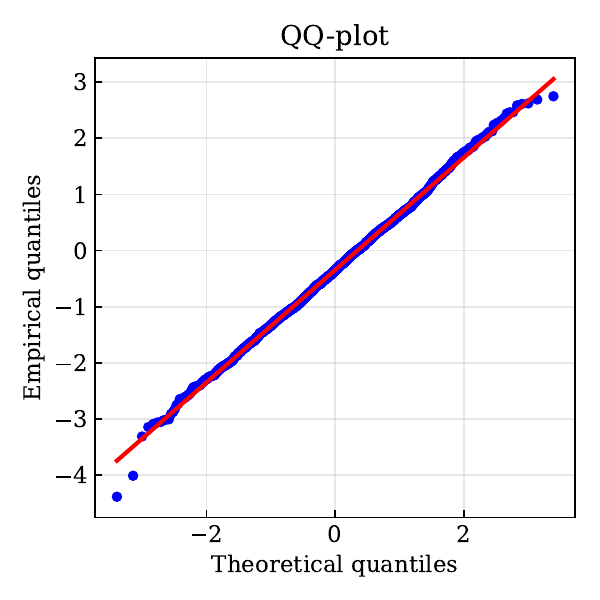}
    \includegraphics[width=0.32\textwidth]{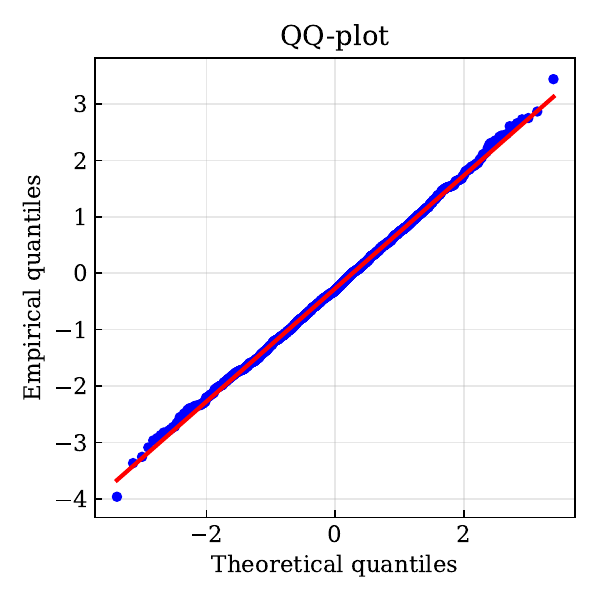}
    \includegraphics[width=0.32\textwidth]{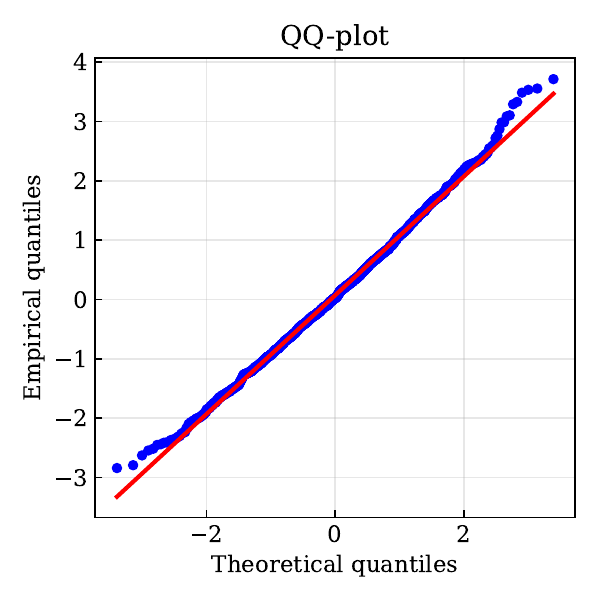}
    \caption{\label{F2}QQ-plot of the empirical distribution of $\hat p_T$ (vertical axis; `empirical quantiles') against the normal distribution (horizontal axis; `theoretical quantiles'), based on $R=2000$ replications, with $n=7$, $M=14$, and $T=4000$. Left panel: $p=0.25$, middle panel: $p=0.50$, right panel: $p=0.75$.}
\end{figure}

\begin{figure}
    \includegraphics[width=0.32\textwidth]{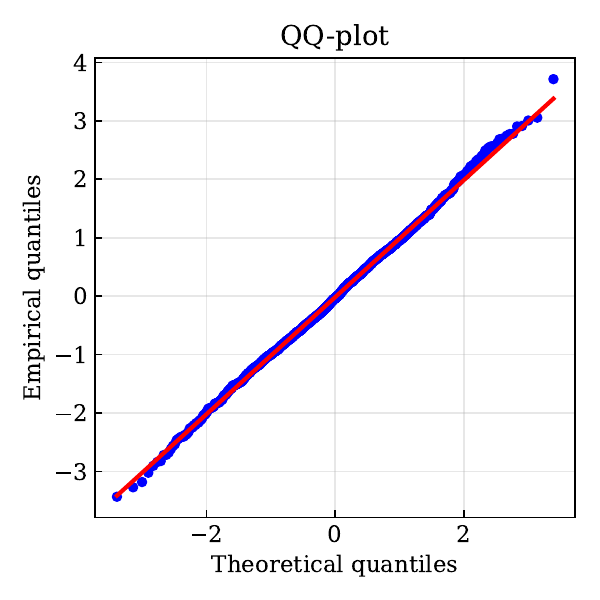}
    \includegraphics[width=0.32\textwidth]{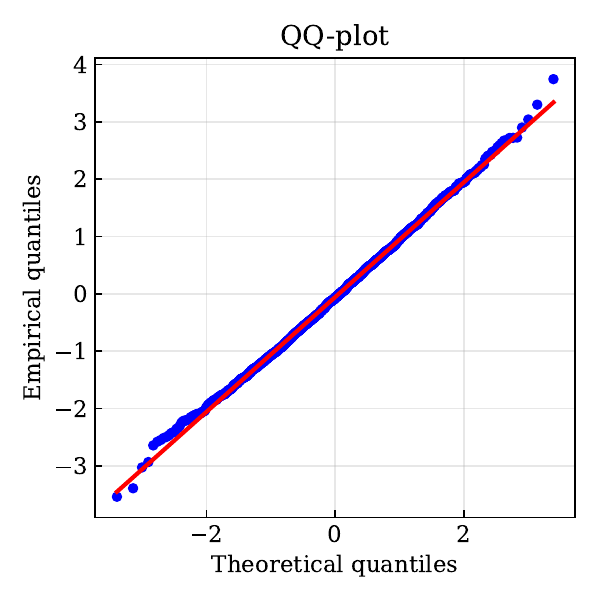}
    \includegraphics[width=0.32\textwidth]{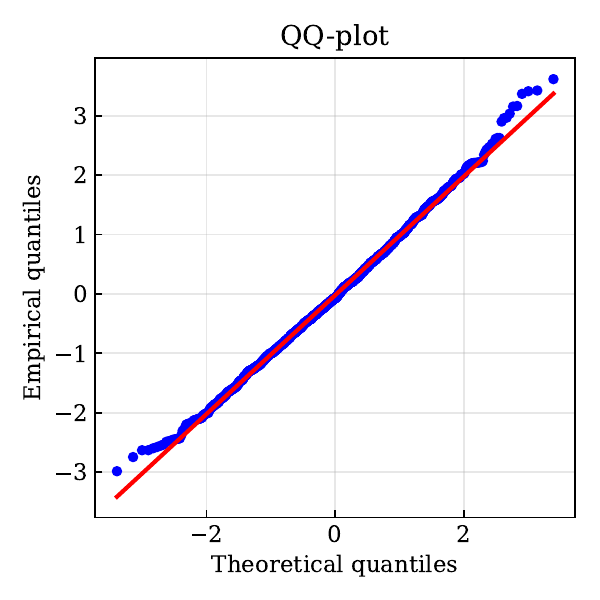}
    \caption{\label{F3}QQ-plot of the empirical distribution of $\bar p_T$ (vertical axis; `empirical quantiles') against the normal distribution (horizontal axis; `theoretical quantiles'), based on $R=2000$ replications, with $n=7$, $M=14$, and $T=4000$. Left panel: $p=0.25$, middle panel: $p=0.50$, right panel: $p=0.75$.}
\end{figure}

We conclude our empirical analysis by comparing the variance of the two estimators. Figure~\ref{F4} displays the quantities $\lambda(p)$ and $\mu(p)$ as functions of $p$. The left panel shows that, for the setting considered, $\lambda(p) > 1$, whereas the right panel shows that $\mu(p) < 1$. 
We performed additional experiments under various alternative parameter configurations, and the same qualitative pattern was observed consistently. We therefore conclude that, with respect to one of the factors entering $\nu(p)$, the estimator $\hat p_T$ performs better, whereas with respect to the other factor $\bar p_T$ has the advantage. Consequently, the overall comparison hinges on their joint effect, which we study next.

 Considering the product $\nu(p)$ of $\lambda(p)$ and $\mu(p)$, as shown in the right panel of Figure \ref{F5}, we observe that for small values of $p$, the estimator $\bar p_T$ has a lower variance than $\hat p_T$, whereas for larger values of $p$, $\hat p_T$ performs slightly better, though the two estimators remain roughly comparable.

We also conducted additional experiments with other values of $n$ and $M$, all of which confirmed the same overall pattern regarding the relative precision of $\hat p_T$ and $\bar p_T$. In particular, when both $n$ and $M$ are increased simultaneously (for example, $n = 15$ and $M = 30$) the performance of the two estimators becomes essentially equivalent.

\begin{figure}
    \includegraphics[width=0.48\textwidth]{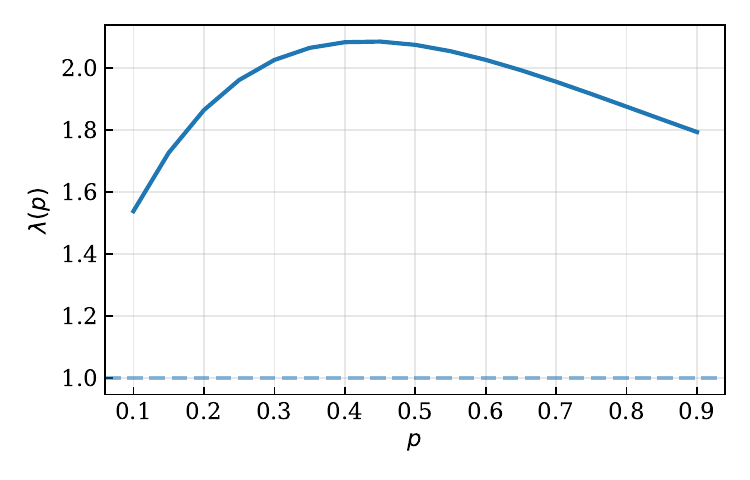}
    \includegraphics[width=0.48\textwidth]{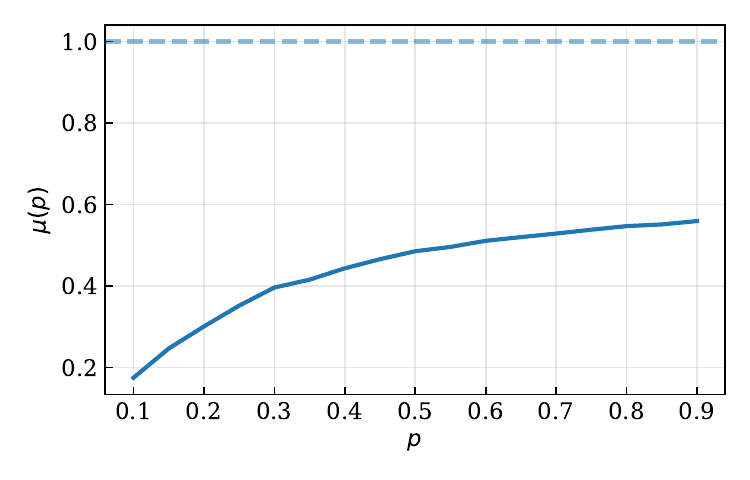}
    \caption{\label{F4}Left panel: $\lambda(p)$ as a function of $p$. Right panel: $\mu(p)$ as a function of $p$. In the setting considered, $n=7$ and $M=14$. The estimates underlying the right panel are based on runs of length $T=4000$ and $R=2000$ replications.}
\end{figure}

\begin{figure}
    \includegraphics[width=0.48\textwidth]{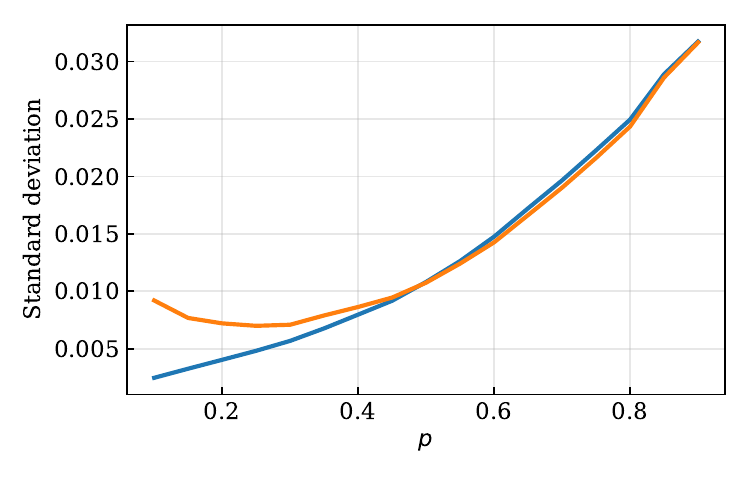}
    \includegraphics[width=0.48\textwidth]{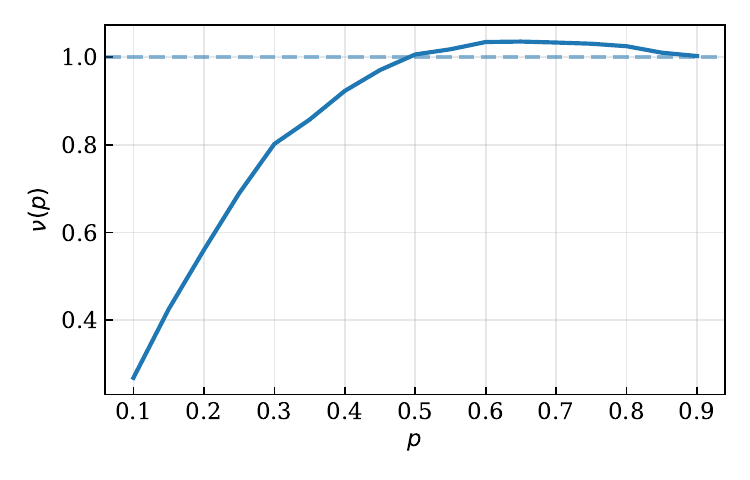}
    \caption{\label{F5}Left panel: the standard deviations of the estimators $\hat p_T$ and $\bar p_T$; orange curve corresponds to $\sqrt{\vphantom{x}\,} \hat p_T$ and blue curve to $\sqrt{\vphantom{x}\,} \bar p_T$. In the setting considered, $n=7$ and $M=14$.  Right panel: $\nu(p)$ as a function of $p$. The estimates are based on runs of length $T=4000$ and $R=2000$ replications.}
\end{figure}

{
}


\begin{thebibliography}{99}%

\bibitem{ACL}
\textsc{Y. A\"it-Sahalia,  J.~A. Cacho-Diaz,} and  {\sc R. Laeven} (2015).
Modeling financial contagion using mutually exciting jump processes.
\textit{Journal of Financial Economics}, {\bf 117}, pp.\ 585-606.



\bibitem{ATH}
{\sc S. Athreya, F. den Hollander,} and {\sc A. Röllin} (2020). Graphon-valued stochastic processes from population genetics. {\it Annals of Applied Probability}, {\bf  31}, pp.\  1724-1745.

\bibitem{BBHM24}
{\sc S. Baldassarri, P. Braunsteins, F. den Hollander,} and {\sc M. Mandjes} (2024). 
Opinion dynamics on dense dynamic random graphs. 
{\em arXiv preprint arXiv:2410.14618.}

\bibitem{BBHM26}
{\sc S. Baldassarri, P. Braunsteins, F. den Hollander,} and {\sc M. Mandjes} (2026). Infection models on dense dynamic random graphs. 
{\em arXiv preprint arXiv:2602.14562.}

\bibitem{BBLS19}
    {\sc F. Ball, T. Britton, K. Y. Leung,} and {\sc D. Sirl} (2019).
     A stochastic SIR network epidemic model with preventive dropping of edges. 
     {\em Journal of Mathematical Biology}, {\bf 78}, pp.\ 1875-1951.

\bibitem{BB22}
    {\sc F. Ball} and {\sc T. Britton} (2022).
    Epidemics on networks with preventive rewiring. 
    {\em Random Structures \& Algorithms}, {\bf 61}, pp.\ 250-297.

\bibitem{BS17}
    {\sc R.\ Basu} and {\sc A.\ Sly} (2017).
    Evolving voter model on dense random graphs. 
    {\em The Annals of Applied Probability} {\bf 27}, pp. \ 1235-1288.

\bibitem{BRA}
 {\sc P. Braunsteins, F. den Hollander}, and {\sc M. Mandjes} (2023).
 A sample-path large deviation principle for dynamic Erd\H{o}s--R\'enyi random graphs. {\it Annals of Applied Probability}, {\bf  33}, pp.\ 3278–3320.

\bibitem{B20}
{\sc T. Britton} (2020). 
Epidemic models on social networks — with inference. 
{\em Statistica Neerlandica}, {\bf 74}, pp.\ 222-241.

\bibitem{BAXV22}
{\sc F. Bu, A.  Aiello, J. Xu},  and {\sc A. Volfovsky} (2022).
Likelihood based inference for partially observed epidemics on dynamic networks, 
{\em Journal of the American Statistical Association}, {\bf 117}, pp.\ 510-526.

 \bibitem{CAU}
{\sc S. Cauchemez} and {\sc N. Ferguson} (2008). Likelihood-based estimation of continuous-time epidemic models from time-series data: application to measles transmission in London. {\it Journal of the Royal Society Interface}, {\bf 5}, pp.\ 885–897.

\bibitem{CD13}
{\sc S. Chatterjee} and {\sc P. Diaconis} (2013). 
Estimating and understanding exponential random graph models.
{\em Annals of Statistics}, {\bf 41}, pp.\ 2428-2461.

\bibitem{ChatterjeeVaradhan2011}
{\sc S.~Chatterjee} and {\sc S.~R.~S.~Varadhan} (2011).
\newblock The large deviation principle for the Erd\H{o}s--R\'enyi random graph.
\newblock \emph{European Journal of Combinatorics}, \textbf{32}, pp.\ 1000--1017.

\bibitem{CRA}
{\sc H. Crane} (2016). Dynamic random networks and their graph limits. {\it Annals of Applied Probability}, {\bf  26}, pp.\ 691–721.


\bibitem{DUF}
{\sc D. Duffie} and {\sc P. Glynn} (2004). Estimation of continuous-time Markov processes sampled at random time
intervals. {\it Econometrica}, {\bf 72}, pp.\ 1773–1808.

\bibitem{D12}
{\sc R. Durrett, J. Gleeson, A. Lloyd, P. Mucha, F. Shi, D. Sivakoff, J. Socolar}, and {\sc C. Varghese} (2012). 
Graph fission in an evolving voter model. 
{\em Proceedings of the National Academy of Sciences}, {\bf 109}, 3682-3687.

\bibitem{ErdosKnowlesYauYin2013}
{\sc L.~Erd\H{o}s}, {\sc A.~Knowles}, {\sc H.-T.~Yau}, and {\sc J.~Yin} (2013).
\newblock Spectral statistics of Erd\H{o}s--R\'enyi graphs I: Local semicircle law.
\newblock \emph{Annals of Probability}, \textbf{41}, pp.\ 2279--2375.


\bibitem{ER}
{\sc P.~Erd\H{o}s} and {\sc A.~R\'enyi} (1960).
\newblock On the evolution of random graphs.
\newblock \emph{Publication of the Mathematical Institute of the Hungarian Academy of Sciences}, \textbf{5}, pp.\ 17-61.


\bibitem{GUT}
{\sc P. Guttorp} (1991). {\it Statistical Inference for Branching Processes}. Wiley, Chichester, UK.

\bibitem{HAN}
{\sc M. Hansen} and {\sc S. Pitts} (2006). Nonparametric inference from the M/G/1 workload. {\it Bernoulli,} {\bf 12}, pp.\ 737–759.



\bibitem{hkm25a}
{\sc R.~S.~Hazra}, {\sc N.~Kriukov}, and {\sc M.~Mandjes} (2026).
\newblock Functional Central Limit Theorem for the principal eigenvalue of dynamic Erd\H{o}s--R\'enyi random graphs.
\newblock \emph{Annals of Applied Probability}, to appear. {\em arXiv preprint arXiv:2407.02686.}


\bibitem{RvdH1}
{\sc R.~van der Hofstad} (2016).
\newblock \emph{Random Graphs and Complex Networks, Volume 1}.
\newblock Cambridge University Press, Cambridge, UK.

\bibitem{RvdH2}
{\sc R.~van der Hofstad} (2023).
\newblock \emph{Random Graphs and Complex Networks, Volume 2}.
\newblock Cambridge University Press, Cambridge, UK.

\bibitem{HolmeSaramaki2012}
{\sc P.~Holme} and {\sc J.~Saram\"aki} (2012).
\newblock Temporal networks.
\newblock \emph{Physics Reports}, \textbf{519}, pp.\ 97-125.

\bibitem{KH14}
{\sc P. Krivitsky} and {\sc M. Handcock} (2014). 
A separable model for dynamic networks. 
{\em Journal of the Royal Statistical Society Series B: Statistical Methodology}, {\bf 76}, pp.\ 29-46.

\bibitem{L82}
{\sc R. Lockhart} (1982). 
On the non-existence of consistent estimates in Galton-Watson processes. 
{\em Journal of Applied Probability}, {\bf 19}, pp.\ 842-846.

\bibitem{Man_Wan}
{\sc M. Mandjes} and {\sc J. Wang} (2026).
{Estimation of on- and off-time distributions in a dynamic Erd\H{o}s--R\'enyi random graph}. 
{\it Advances in Applied Probability}, to appear. {\em arXiv preprint 	arXiv:2401.14531}.

\bibitem{MT12}
{\sc S. Meyn} and {\sc S. Tweedie} (2012). {\em Markov Chains and Stochastic Stability.} Springer, New York, USA.

\bibitem{RAV}
{\sc L. Ravner, O. Boxma,} and {\sc M. Mandjes} (2019). Estimating the input of a 
Lévy-driven queue by Poisson sampling
of the workload process. {\it Bernoulli}, {\bf 25}, pp.\ 3734–3769.

\end{thebibliography}
\end{document}